\documentclass[11pt]{amsart}
\usepackage{amssymb}
\usepackage{amsmath}
\usepackage{bm}
\usepackage{ifpdf}
\usepackage{graphics}
\usepackage[arrow,matrix,tips]{xy}
\usepackage{geometry} 
\usepackage{booktabs} 
\usepackage{array} 
\usepackage{paralist} 
\usepackage{verbatim} 
\usepackage{subfigure} 

\usepackage{geometry} 

\usepackage{fancyhdr} 
\newtheorem{thm}{Theorem}[section]
\newtheorem{con}{Conjecture}
\newtheorem{cor}[thm]{Corollary}
\newtheorem{lem}[thm]{Lemma}
\newtheorem{prop}[thm]{Proposition}
\theoremstyle{definition}
\newtheorem{defn}[thm]{Definition}
\newtheorem{defns}[thm]{Definitions}

\theoremstyle{remark}
\newtheorem{rem}[thm]{Remark}
\numberwithin{equation}{section}

\begin{document}
\title{Large flats in the pants graph}
\author{Jos\'e L. Est\'evez}
\date{} 
\address{UNED. Departamento de Matem\'aticas Fundamentales, C/ Senda del Rey, 9. 28040 Madrid. Spain.}
\email{jestevez@mat.uned.es}

\thanks{}
\subjclass[2010]{Primary 57M50; Secondary 05C12}
\keywords{Pants complex. Weil-Petersson metric. }
\begin{abstract} This note is about the geometry of the pants graph $\mathcal{P}$, a natural simplicial graph associated to a finite type topological surface $\Sigma$ where vertices represents pants decompositions. The main result in this note asserts that for a multicurve $Q$ whose complement is a number of subsurfaces of complexity at most $1$ the corresponding subgraph $\mathcal{P}_{Q}$ is totally geodesic in $\mathcal{P}$ previously considering this as a metric space assigning length one to each edge. As a consequence of this result we make explicit  the existence of large flats in the pants graph.
\end{abstract}

\maketitle.

\section{Introduction}

Let $\Sigma$ be a compact orientable surface of genus $g$ with $r$ boundary components and such that $3g-3+r>0$. The pants graph $\mathcal{P}$ is a graph associated the surface $\Sigma$, was introduced in \cite{HT} and further studied in \cite{H}. The number $3g-3+r>0$, is the complex dimension of its Teichm\"uller space $\mathcal{T}$ of $\Sigma$ and it is also called the complexity of $\Sigma$. By assigning length one to the edges in $\mathcal{P}$ we obtain a metric space $(\mathcal{P},d)$ and it reveals as a combinatorial model of the Teichm\"uller  space with the Weil-Petersson metric \cite[Theorem 1.1]{Br1}, specifically there is a quasi-isometry between the $0$-squeleton of $\mathcal{P}$ and $\mathcal{T}$ with the WP metric.

In \cite{APK1, APK2} the authors study geometric properties of the \emph{pants graph} $\mathcal{P}$ by considering projections $\mathcal{P}\to \mathcal{P}_{Q}$ over the subgraph corresponding to a $2$-handle multicurve. They exhibit $\mathcal{P}_{Q}$ as a totally geodesic subgraph of $\mathcal{P}$ and as a consequence the existence of isometric embeddings of $\mathbb{Z}^{2}$ into $\mathcal{P}_{Q}$. 

The existence of maximal size quasi-flats in $\mathcal{P}$ is known, (quasi-isometric embeddings of $\mathbb{Z}^{n}$ in $\mathcal{P}$ with maximal $n$), see \cite{B-M}. In this note we prove the following.

\begin{thm}\label{T1} Let $\Sigma$ be a compact, connected and orientable surface, and denote by $Q$ an arbitrary $n$-handle multicurve on $\Sigma$. Then, $\mathcal{P}_{Q}$ is totally geodesic in  $\mathcal{P}$.
\end{thm}

from which we deduce the existence of large flats in $\mathcal{P}$ in section \ref{LF} and that partially supports the following \cite{APK1, APK2}:

\begin{con} Given two compact and orientable surfaces $\Sigma_{1}$ and
$\Sigma_{2}$, the image of a simplicial embedding $\phi:\mathcal{P}(\Sigma_{1})\to \mathcal{P}(\Sigma_{2})$ is totally geodesic.
\end{con}

In this sense it should be mentioned \cite{ALPK} where by considering slightly modified modified graphs, the authors conclude that the Conjecture is true for the $6$-holed sphere.

The Weil-Petersson completion $\overline{\mathcal{T}}$ of $\mathcal{T}$ is obtained by including noded surfaces and $\overline{\mathcal{T}}$ becomes a stratified space, each stratum being a product of lower dimensional Teichm\"uller spaces and WP metric on them is the extension of the WP metric on $\mathcal{T}$. Given a multicurve $C$ on $\Sigma$ we obtain $\mathcal{T}_{C}\subset \partial \mathcal{T}$ by pinching the curves on $C$, the stratum $\mathcal{T}_{C}$ and its completion $\overline{\mathcal{T}}_{C}$ are totally geodesic in $\overline{\mathcal{T}}$ that is known to be a unique geodesic space \cite[Theorem 13]{W1}, \cite{Y}. A maximally noded surface corresponds to a pants decomposition of $\Sigma$ and it lies in a  $0$-dimensional stratum in $\overline{\mathcal{T}}\setminus \mathcal{T}$ \cite{C-P}.

If we see $\Sigma$ as a disjoint union of $g$ $1$-holed torus, $\lfloor \frac{g+r}{2} \rfloor-1$ $4$-holed spheres and eventually a $3$-holed sphere, this decomposition may arise as the complementary of certain $(g+\lfloor \frac{g+r}{2} \rfloor-1)$-handle multicurve $Q$ (eventually ignoring the three-holed sphere) and we deduce from the main theorem that the correspnding product of Farey graphs is totally geodesic in the pants graph, as a corollary we find an isometric embedding $\mathbb{Z}^{(g+\lfloor \frac{g+r}{2} \rfloor-1)}\to \mathcal{P}$ corresponding to an embedding $\mathbb{R}^{(g+\lfloor \frac{g+r}{2} \rfloor-1)}\to \overline{\mathcal{T}}$. The number of complexity $1$ subsurfaces thus obtained $\lfloor \frac{3g+r-2}{2} \rfloor$ is the geometric rank of $\overline{\mathcal{T}}$, \cite{B-M, Br-F} which translates to $\mathcal{P}$ by the quasi-isometry exhibited in \cite{Br1}.

\mathstrut

\begin{figure}[h]\label{Farey}
\begin{center}
\scalebox{.4}{\includegraphics{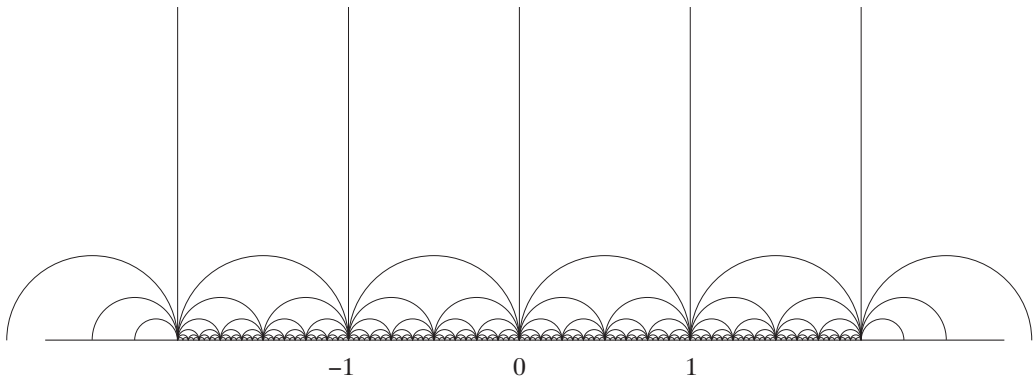}}
\end{center}
\caption{The figure shows the Farey graph as embedded in half plane model of $\mathbb{H}^{2}$. The vertex set is $\mathbb{Q}\cup \{1/0\}$ and represents the completion of $\mathcal{T}$. Two vertices $p/q$ and $r/s$ are joined by an edge when $|ps-rq|=1$.}
\label{Multi1}
\end{figure}

\section{Definitions}

\subsection{Curves and multicurves}\label{mc}

For the surface $\Sigma$ we shall denote $\kappa(\Sigma)$ the integer $3g-3+r$ and call it the complexity of $\Sigma$. This number denotes the complex dimension of the Teichm\"uller space of $\Sigma$.

A \emph{multicurve} is a collection of pairwise disjoint isotopy classes of simple closed curves, none of them isotopic to a boundary component.

We say a multicurve $\omega$ has rank $n\geq 2$, if $|\omega|=\kappa(\Sigma)-n$ for some non-negative integer $n$. A rank $n$ multicurve $Q$ is a $n$-\textit{handle multicurve} if the complement of every simple representative of $Q$ contains $n$ complexity $1$ components each of them with every boundary component trivial in the homology group $H_{1}(\Sigma,\partial \Sigma;\mathbb{Z})$.

In this paper we are concerned with $n$-handle multicurves where $2\leq n \leq \lfloor \frac{3g+r-2}{2} \rfloor$, where the upper bound represents the maximum number of pairwise disjoint subsurfaces of complexity $1$ contained in $\Sigma$.

In order to perform surgery constructions we refer to the representatives of these classes. In this sense we introduce the following definition that avoids to refer us to a geometrical structure on $\Sigma$ and makes easier t deal with those conditions we need.

\begin{defn} Let $Y$ be a subsurface of $\Sigma$ of complexity $1$. Then:
\begin{enumerate} 
\item We say that a simple closed curve $a$ intersecting $Y$ is \emph{tight} in $Y$ if the number of intersection points of $a$ with $\partial Y$ is minimal among all those curves in its isotopy class.
\item For a curve $a$ that is tight in $Y$ but $a\not\subset Y$, we shall call each component of $a\cup Y$ an \emph{essential arc} in $Y$.
\item We say that two curves $\alpha, \beta$ intersecting $Y$ are \emph{tight} in $Y$ if each of them is tight in $Y$ and there is no region $R\subset Y$  homeomorphic to a disc which is limited by an arc $a\subset \alpha$ another arc $b\subset \beta$ and eventually another $d\subset \partial Y$.
\item Two arcs $a$ and $b$ contained in $Y$ are \emph{tight} if they are the intersections of two curves $\alpha$ and $\beta$ that are tight in $Y$. 
\end{enumerate}
\end{defn}

\begin{rem} The existence of a region $R$ homeomorphic to a disc being limited by two arcs $a\subset \alpha, b\subset \beta$ and eventually $d\subset \partial Y$, implies the existence of an isotopy of $(Y,\partial Y)$ that moves $\alpha$ and eliminates two intersection points in $\alpha \cap \beta$ in the first case or one intersection point in $(\alpha\cap\beta)\cup (\alpha\cap \partial Y)\cup (\beta \cap \partial Y)$. 
\end{rem}

\subsection{Pants decomposition} 

A \emph{pants decomposition} of a surface $\Sigma$ is a maximal multicurve $\nu \subset \Sigma$. The resulting decomposition $\Sigma \setminus \nu$ consists on $2g-2+r$ $3$-holed spheres. The number of curves that define a pants decomposition in $\Sigma$ is $\kappa(\Sigma)$.

\subsection{Elementary moves} Given two pants decompositions $\nu_{1}$ and $\nu_{2}$ we say that they are related by an \emph{elementary move} if one is obtained from the other by changing one curve $a$ by another $b$ that intersects $a$ minimally: if $\Sigma \setminus (\nu \setminus a)$ contains a one holed torus $b$ intersects $a$ in one single point and this elementary move is of the \emph{first kind}, otherwise if $\Sigma \setminus (\nu \setminus a)$ contains a four holed sphere $b$ intersects $a$ in two points and this elementary move is of the \emph{second kind}. We shall write $\nu_{1}\to \nu_{2}$ to denote that $\nu_{1}$ and $\nu_{2}$ are related by an elementary move.

\begin{rem} Observe that both the $1$-holed torus in one case as the $4$-holed sphere in the other,  are isotopic to a regular neigbourhood of $a\cup b$, where $a$ and $b$ are the curves involved in an elementary move.
\end{rem}

\subsection{Pants graph} The \emph{pants graph} $\mathcal{P}$ is the graph whose vertices are pants decompositions and two vertices are joined by an edge whenever the are related by an elementary move. The pants graph is the $1-$squeleton of the pants complex $\mathcal{PC}$ that is known to be connected and simply connected \cite{H,HT}. We endow the pants graph with a metric by assigning length $1$ to each edge. A \emph{geodesic arc} in $(\mathcal{P},d)$ is an edge path joining two vertices that minimizes distance. In general a \emph{geodesic} $g\subset \mathcal{P}$ is a set of consecutive vertices $\{v_{i}:i\in J\}$ together with the corresponding edges such that for each pair $(v_{r},v_{s})$ there is a geodesic arc joining them and lying in $g$. 

A subgraph $\mathcal{G}$ of $\mathcal{P}$ is \emph{totally geodesic} if any geodesic arc with its end points in $\mathcal{G}$ is contained in $\mathcal{G}$. A subgraph $\mathcal{G}$ of $\mathcal{P}$ is \emph{convex} when any two points of $\mathcal{G}$ may be joined by a geodesic arc contained in $\mathcal{G}$. Totally geodesic implies convex but the converse is not true, for instance if we look at the Farey graph $\mathcal{F}$ as in figure \ref{Farey} and consider the subgraph $\mathcal{G}$ that includes all the vertices in the interval $[-1,1]$, and as edges only those in $\mathcal{F}$ connecting these vertices.

\subsection{Subgraphs of the pants graph}

Given a multicurve $Q\subset \Sigma$ there is a set of pants decompositions of $\Sigma$ having these multicurve as a subset. If we consider the subgraph $\mathcal{P}_{Q}\subset \mathcal{P}$ formed by this set of vertices together with the corresponding edges in $\mathcal{P}$ we have a subgraph representing all the pants decomposition that contain the curve $Q$. The distance measured inside $\mathcal{P}_{Q}$ is noted $d_{Q}$. Also given a subsurface $Y\subset \Sigma$ we may refer to the pants graph $\mathcal{P}_{Y}$ of $Y$. Observe that for any two different multicurves $Q$ and $Q'$ having $Y$ as complementary surface of complexity grater than $0$, there are two copies of $\mathcal{P}_{Y}$ embedded in $\mathcal{P}$.

\subsection{Subsurface projections} Let $Y$ be a essential subsurface of complexity $1$ and let $a\subset \Sigma$ be a curve that is tight with respect to $Y$ and let $\{a_{i}:i=1,\cdots,k\}$ be the finite set of connected components of $a\cap Y$. The projection $\pi_{Y}(a)$ is the set of curves $\{c_{i}:i=1,\cdots,n\}$ where $c_{i}$ is the only curve contained in $Y$ that does not intersect $a_{i}$. In case of $a\cap Y=\varnothing$ we set $\pi_{Y}(a)=\varnothing$. This definition extends naturally to the case of $\omega$ being a multicurve, thus resulting $\pi_{Y}(\omega)$ as a finite set. 

More generally, if $Q$ is a $n$-handle multicurve and $Y_{1},\dots,Y_{n}$ the family of all complexity $1$ subsurfaces contained in $\Sigma \setminus Q$ we may define $\pi_{Q}(\omega)$ for any multicurve $\omega$ as the set $\{a_{i}\cup Q: a_{i}\in \pi_{Y_{i}}(\omega)\}$. So those $Y_{i}$ such that $\omega \cap Y_{i}=\varnothing$ do not contribute with any curve $a_{i}\subset Y_{i}$ to the projection, in particular if $\omega \cap (Y_{1}\cup \cdots \cup Y_{n})=\varnothing$ then $\pi_{Q}(\omega)=Q$.

\begin{rem} We shall consider projections $\pi_{Y}(a)$ that are not isotopic to $\partial Y$.
\end{rem}

\subsection{Seams and waves} In case that $Y\subset \Sigma$ is a $4-$holed sphere we call a \emph{wave} an essential an arc $w\subset Y$ with its end points lying at the same boundary component of $Y$. In case of an essential arc $s\subset Y$ that has its end points at different components we call it a \emph{seam}.

\subsection{Partial projections} Let $s$ be a seam in a $4$-holed sphere $Y$ and let $\delta_{1}$ and $\delta_{2}$ be the boundary components of $Y$ in which the end points of $s$ lie. If $N$ is a regular neighborhood of $\delta_{1} \cup s$ then $a=\partial N$ is a wave having their end points at $\delta_{2}$ and such that $\pi_{Y}(a)=\pi_{Y}(s)$. This we call the \emph{partial projection} of $s$ over $\delta_{1}$. The same thing we may do for the component $\delta_{2}$ having another wave $b$ that verifies $\pi_{Y}(b)=\pi_{Y}(s)$.

\subsection{Notation} Al along this note we shall denote by $\Sigma$ an orientable surface of genus $g$ with $r$ boundary components and $Q\subset \Sigma$ is an $n$-handle multicurve. Given two curves $\alpha,\beta \subset \Sigma$ use the expression $|\alpha \cap \beta|$ or $|\alpha\cap \beta \cap Y|$ to mean the cardinality of the set of intersection points of the curves $\alpha$ and $\beta$ or that of this set inside a subsurface $Y\subset \Sigma$.  By $\iota (\alpha,\beta)$ we denote the \emph{geometric intersection number} of $\alpha$ and $\beta$, that is, the minimum number of intersection points in their homotopy classes.

\mathstrut

\section{$2$-handle multicurves}

 Even in the case of a $2$-handle multicurve $Q\subset \Sigma$ the projection $\pi:\mathcal{P}\to \mathcal{P}_{Q}$ we may not assure that $\pi$ is distance non-increasing as the authors ask in \cite{APK2}. In the example shown in figure \ref{Multi1} we have considered a $2$-handle multicurve $Q$ whose complexity $1$ subsurfaces $Y_{1}$ and $Y_{2}$ are $4$-holed spheres. 

\begin{figure}[h]
\begin{center}
\scalebox{0.27}{\includegraphics{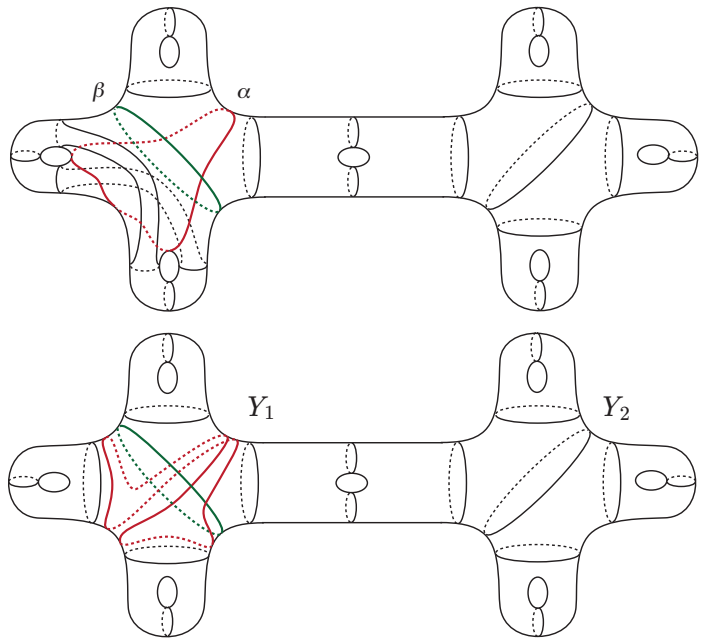}}
\end{center}
\caption{The top figure shows the closed genus $7$ surface $\Sigma$ with two pants decompositions $\nu_{0}$ and $\nu_{1}$ related by an elementary move that is realized by the curves $\alpha \in \nu_{0}$ and $\beta \in\nu_{1}$, where $|\alpha\cap \beta|=2$. The figure below shows the $4$-holed spheres $Y_{1}$ and $Y_{2}$ that are disjoint of the $2$-handle multicurve $Q$. Projections $\omega_{0}=\pi_{Q}(\nu_{0})$ and $\pi_{Q}(\nu_{1})$ are at distance $2$ and there is no $\omega \in \pi_{Q}(\nu_{1})$ at lower distance from $\omega_{0}$.}
\label{Multi1}
\end{figure}

To prove Theorem \ref{T1} we analyze the intersections of curves that realize an elementary move with the complexity $1$ subsurfaces.

\section{Intersection of curves on a $1$-holed torus}

\begin{defn} Let $\alpha$ and $\beta$ be two curves related by an elementary move and intersecting a $1$-holed torus $Y\subset \Sigma$ in one single essential arc each, say $a\subset Y$ and $b\subset Y$. We say that the end points of $a$ and $b$ are \emph{linked} if they are linked in $\partial Y$.
\end{defn}

\begin{lem}\label{lk} Let $a$ and $b$ two different essential arcs contained in a $1$-holed torus $Y$ and having their end points at $\partial Y$. If $a$ and $b$ are tight in $Y$ then their end points are linked. 
\end{lem}
\begin{proof} Call $x_{1},x_{2}\in \partial Y$ the couple of end points of $a$ and $y_{1}, y_{2}$ those of $b$. Consider $A=Y\setminus a$, then as $a$ is essential it is homeomorphic to an annulus and let $A_{1}$ and $A_{2}$ be their boundary components. Let $a_{1}\subset A_{1}$ and $a_{2}\subset A_{2}$ be the arcs that correspond to $a$ and call $x'_{1},x'_{2}\in a_{1}$ and $x''_{1},x''_{2}\in a_{2}$ the end points corresponding to $x_{1}$ and $x_{2}$. 

Any essential arc $b\subset Y$ (may be intersecting $a$ or not) that is not isotopic to $a$ reflects in $A$ as a (finite) collection of disjoint arcs $b_{i}$ with their end points fixed, each one of them connecting the boundary components of $A$. As $b\subset Y$ has one single connected component and has its end points at $\partial Y$, results that only two arcs, say $b_{1}$ and $b_{r}$ have one end point each at $A_{1}\setminus a_{1}$ and $A_{2}\setminus a_{2}$, this points are $y_{1}\in A_{1}$ and $y_{2}\in A_{2}$. All the rest of end points of arcs $b_{i}$ including one of $b_{1}$ and one of $b_{r}$ lie on $a_{1}\cup a_{2}$ and correspond to points of $a\cap b$. To see that the couples $(x_{1},x_{2})$ and $(y_{1},y_{2})$ are linked identify the segments $a_{1} \sim a_{2}$ in such way that $x'_{1} \sim x''_{1}$ and $x'_{2} \sim x''_{2}$. 

This identification is equivalent to attaching a $2$-dimensional handle $[0,1]\times[0,1]$ at $A$ in such way that $[0,1]\times \left\{0\right\}$ identifies with $a_{1}$ and $[0,1]\times \left\{1\right\}$ identifies with $a_{2}$. We may distinguish the points $\left\{0\right\}\times \left\{1/2\right\}$ and $\left\{1\right\}\times \left\{1/2\right\}$ as corresponding to $x_{1}$ and $x_{2}$. The resulting boundary component is divided in four segments, two of them corresponding to the attached handle and separating the other two.
\end{proof}

\begin{cor}\label{3atm} Let $Y$ be a $1$-holed torus. Let $\alpha,\beta \subset \Sigma$ be two curves that are tight in $Y$, each intersecting $Y$ in a finite set of arcs and both intersecting in one point inside $Y$. Let $A$ be the connected component of $(\alpha \cup \beta)\cap Y$ that contains the intersection point and let $N$ be a regular neighborhood of $A$ then the arcs $d_{i}, i=1,\dots ,4$ are two by two related by an isotopy of arcs that keeps their end points in $\partial Y$.    
\end{cor}
\begin{proof} Let $a=\alpha \cap Y$ and $b=\beta \cap Y$ be the arcs that contain the intersection point $\alpha \cap \beta \subset Y$. Consider also the arcs $d_{i},i=1,\dots ,4$. As the curves $\alpha$ and $\beta$ are tight in $Y$, no arc $d_{i}$ is isotopic to $a$ or $b$, then if $d_{i}$ and $d_{j}$ are isotopic in $Y$ with their end points in $\partial Y$ the isotopy may be considered also in $Y\setminus N$. So in case of three arcs $d_{i}$ being isotopic there would be one of them between the other two, so contradicting $N$ being a regular neighborhood. Assuming the four arcs $d_{i}$ are isotopic would take us to $a$ and $b$ being isotopic too, which is false.

Each pair of of arcs have their end points linked and constructing the projection over $\mathcal{P}_{Y}$ of one of them consists on joining its end points by one arc parallel to $\partial Y$. 

It results that $\pi_{Y}(a),\pi_{Y}(b)$ and $\pi_{Y}(d_{i})$ are the vertices of a triangle on the Farey graph $\mathcal{P}_{Y}$. In case of having a fourth isotopy class different of the former we would have obtained four vertices in    $\mathcal{P}_{Y}$ at distance $1$ any two of them, which is impossible.
\end{proof}

\begin{lem}\label{prt} Let $\alpha$ and $\beta$ be two curves in $\Sigma$ realizing an elementary move $\nu_{0}\to \nu_{1}$ and tight in $Y$. Let $A\subset Y$ be a connected component of $(\alpha\cup \beta)\cap Y$ and let $N$ be a regular neighborhood of $A$. Then there is a curve $d\subset \partial N$ such that $d_{Y}(\pi_{Y}(a),\pi_{Y}(d))\leq 1$ for any arc or curve $a\subset \left\{\alpha\cap Y \right\}\cup \left\{\beta\cap Y \right\}$
\end{lem}
\begin{proof} Let $A$ be the connected component in the statement. Assume first that $A$ just contains at most one intersection point of $\alpha \cap \beta$ and none of the curves $\alpha$ or $\beta$ is contained in $Y$, then consider a regular neighborhood $N\subset Y$ of $A$. 
With this initial assumption it results that any arc $d\subset \partial N\subset Y$ is essential as all the arcs $\alpha \cap Y$ and $\beta\cap Y$ contained in $A$ are tight. For any arc $a\subset \left\{\alpha\cap Y \right\}\cup \left\{\beta\cap Y \right\}$ it results by lemma \ref{lk} that the end points of $a$ and $d$ are linked. In case that one of the former curves, say $\alpha$ is contained in $Y$ we have $\iota (\alpha, \beta)=1$, the neighborhood $N$ has two essential arcs $d_{1},d_{2}\subset Y$ each projecting to $\alpha$ and then $d_{Y}(\pi_{Y}(d_{i}),\pi_{Y}(\beta))=1$. 

It rests to consider the case in which $\alpha$ and $\beta$ intersect in two points inside $Y$. As in the previous case we ask to $\partial N$ to have some arc to be essential in $Y$, and proceed as before, but as in that case we might assure the existence of such an arc without appealing $\alpha$ and $\beta$ to be related by an elementary move. If we drop this condition here, it may happen that $\partial N$ has no such arc and then it results, after an isotopy, $N$ having a holed torus embedded. This would contradict the fact of $N$ being a $4$-holed sphere when there is an elementary move relating two curves that intersect in two points. 
\end{proof}

\mathstrut

\section{Intersection of curves on a $4$-holed sphere}

In the particular case of $Y$ being a $4$-holed sphere we precise of some  previous definitions. All along this section $Y$ will denote a $4$-holed sphere.

\begin{defns} 
\begin{enumerate}
\item Consider a seam $s$ together with a curve $c$ contained in $Y$ such that $\iota(s\cap c)=2$, then we say that $s$ and $c$ are a \emph{special couple} in $Y$.
\item For two different seams $s,t \subset Y$, we shall call the \emph{number common boundaries} of $s$ and $t$, the number of connected components of $\partial Y$ containing one end point of each of $s$ and $t$.
\end{enumerate}
\end{defns}

An example of special couple is shown in figure \ref{Multi1} above. This is the intersection of the connected component of $\alpha\cup \beta$ that contains their intersection points with $Y_{1}$.

\begin{rem} Given two tight curves $\alpha,\beta \subset Y$, then $d_{Y}(\alpha,\beta)$, when this number is no greater than $2$ may be measured in terms of $\iota(\alpha,\beta)$, as all elementary moves allowed for curves inside a $4$-holed sphere are of the second kind. 
\end{rem}

\begin{lem}\label{int}
\begin{enumerate}
\item Let $s,t\subset Y$ be two different seams that are tight in $Y$, then $\iota(\pi_{Y}(s)\cap t)=2\iota(s\cap t)+j$ where $j$ is the number of common bondaries of $s$ and $t$. 
\item Let $s, c \subset Y$ be a seam and a curve respectively, then $\iota(\pi_{Y}(s)\cap c)=2\iota(s,c)$
\end{enumerate}
\end{lem}
\begin{proof} 
\begin{enumerate}
\item Let $s$ and $t$ be two representatives that realize $\iota(s,t)$ and let $N\subset Y$ be a regular neighborhood of the complex formed by $s$ and the boundary components $\delta_{1}$ and $\delta_{2}$ where $s$ has its end points. Observe that $\partial N=\pi_{Y}(s)$.

We may select $N$ small enough to have $\iota (\pi_{Y}(s),t)=2|s\cap t|+j$, so it is at most this number. To see that it is also the lowest number assume that it is greater that $0$.

When $2|s\cap t|+j=1$ is $j=1$ and if $P_{1}\sqcup P_{2}=Y\setminus \pi_{Y}(s)$  then $t$ has one end point at $P_{1}$ and the other at $P_{2}$, and as $\pi_{Y}(s)$ separates this two subsurfaces, the number cannot be lowered.

When $2|s\cap t|+j=2$ either $j=0$, in which case $t$ has its end points at, say $P_{1}$, but also a subarc of it (that intersects $s$), lies in $P_{2}$ and for the same reason as before the number can not be lowered; or $j=2$ and $t$ has the end points at $P_{2}$ and a subarc of it lies on $P_{1}$, again the number can not be lowered.

In the remaining cases we have a collection of arcs $t_{i}$ (subarcs of $t$) with their end points at $\partial N$ and eventually some of them (at most two) with one end point at $\partial N$ and the other at $\partial Y$.

Diminishing the number $2|s\cap t|+j$ relies on finding an isotopy of $(Y,\partial Y)$ that takes $t$ to have a lower intersection number with $\partial N$. We may assume that this isotopy applies $N$ into itself and then applies  some arc $t_{i}$ with its end points in $\partial N$ inside $N$. The arc $t_{i}$ together with a subarc $d_{i}\subset \partial N$ determined by its end points would bound a disc, so by the regularity of $N$ we may extend $t_{i}$ to $s$ or any of $\delta_{1}$ or $\delta_{2}$ to bound a disc with any of this which is against $s$ and $t$ being different, tight or with the lower intersection. 
\item The proof of this point is like the previous one with the particularity that in this case the arcs of $c$ outside the neighborhood $N$ of $s$ all have their end points at $\partial N$.
\end{enumerate}
\end{proof}

\mathstrut

\begin{lem}\label{seam} Let $Y$ be a four-holed sphere. For any curve $a\subset \Sigma$ such that $a \cap Y$ has a single connected component there is a seam $s$ such that $\pi_{Y}(s)=\pi_{Y}(a)$.
\end{lem}
\begin{proof} Given the curve $a$ we shall construct an associated seam $s$ as follows. In case that $a$ is itself a seam then $a=s$. So we assume that either $a\subset Y$ or $a\cap Y $ is a wave. In the first case we may construct a wave $w$ by deforming $a$ isotopically to make an arc $a'\subset a$ into one component of $\partial Y$ and as a result $\pi_{Y}(a)=\pi_{Y}(w)$. 

We then assume that $a$ is a wave. The simple closed curve $\pi_{Y}(a)$ is separating and so the wave $a$ is also separating. Consider the component $A$ of $Y\setminus a$ that is homeomorphic to an annulus and call $b$ the component of $\partial Y$ contained in $A$. We see as $a$ can be deformed by a deformation retraction onto the complex formed by $b$ and an arc $s$ that joins $b$ to the other component of $\partial A$ and having empty intersection with $a$. The arc $s$ is the seam we are looking for and $\pi_{Y}(s)=\pi_{Y}(a)$.
\end{proof}

\mathstrut

\begin{defn} The seam obtained in the last Lemma is called \emph{associated seam} to $a$.
\end{defn}

\begin{rem} If $s$ is a seam and $c$ a curve and they are special with respect to $Y$, then one of the associated seams to $c$, say $t$, has its end points at the same components that $s$ does, that is, the number of common boundaries for $s$ and $t$ is $2$. In this case we have $d_{Y}(\pi_{Y}(s),c)=2$.
\end{rem}

\mathstrut

\begin{lem}\label{prs} Let $s$ be a seam and let $b$ be any curve, wave or seam. Assume that $s$ and $b$ are tight with respect to $Y$, $s\cap b=\varnothing$ and in case of $b$ being a seam the number of common boundaries of $s$ and $b$ is at most $1$. Then $d_{Y}(\pi_{Y}(s),\pi_{Y}(b))\leq 1$.
\end{lem}
\begin{proof} If $b$ is a curve or a wave it separates $Y$ so any of its associated seams, say $t$ lies in one of the components of $Y\setminus b$ and it follows that $s\cap t=\varnothing$. We may select the associated seam $t$ lying inside the component that does not contain $s$ to ensure that their number of common boundaries is at most $1$, then by the first part of lemma \ref{int} $\iota(\pi_{Y}(t),s)\leq 1$ and by the second part of the same lemma $\iota(\pi_{Y}(t),\pi_{Y}(s))\leq 2$, which is to say $d_{Y}(\pi_{Y}(s),\pi_{Y}(b))\leq 1$. 
\end{proof}

\mathstrut

\begin{lem}\label{ml} Let $\alpha$ and $\beta$ be two curves in $\Sigma$ related by an elementary move, assume $\alpha$ and $\beta$ are tight in $Y$. Let $A$ be a connected component of $(\alpha \cup \beta)\cap Y$ that is not a special couple and consider a regular neighborhood $N$ of $A$. Then there exists an arc $\delta \in \partial N$ such that $d_{Y}(\pi_{Y}(\delta),\pi_{Y}(a))\leq 1$ for any arc or curve $a \subset \left\{\alpha\cap Y\right\}\cup\left\{\beta\cap Y\right\}$. 
\end{lem}

\begin{proof} Let $A$ be a component of $(\alpha \cup \beta)\cap Y$. The regular neighborhood $N$ of $A$ is a genus $0$ surface as it is embedded in $Y$. The surface $N$ may be contractible, homotopy equivalent to a circle or to a wedge of two circles, depending on the number of intersection points of $\alpha \cap \beta$ inside $Y$ and also on the separation property that the components of the former curves in $A$ satisfy with respect to $Y$. 

To see this more precisely consider the connected $1$-complex $A\subset Y$ and assume that it contains al least one intersection point of $\alpha \cap \beta$.  Delete those arcs in $A$ joining intersection points with some boundary component, the resulting complex $B$ contains the set of intersection points from $\alpha \cap \beta$ together with an arc joining them when this set contains two points. To have valence $4$ would imply that no deleting was made, but this is not possible unless $\alpha, \beta \subset Y$. It clearly results that $B$ has the homotopy type of either a point, a circle or a wedge of two circles. In any case the complex $B$ allows two boundary components of $\partial Y$ to be connected by an arc $d$ that avoids any intersection with $B$ and so also with $A$. In case that $B$ has the homotopy type of a wedge of two circles we have deleted from $A$ a couple of arcs as the vertices of $B$ have valence $3$, so one of the former curves, say $\alpha$ is contained in $Y$ and the arcs in $A$ to be deleted can not end at different components of $\partial Y$ as $A$ is not special. It follows that $d$ does not share ends at different components with any arc in $\alpha \cap A$ or $\beta \cap A$ and the result follows from Lemma \ref{prs}.  
\end{proof}

\mathstrut

\section{Geodesic paths in $\mathcal{P}$}

The main result in this section is concerned with geodesic paths in the pants graph $\mathcal{P}$. Theorem \ref{TP} is crucial to prove the main Theorem in this paper, and before proving it we deal firstly with arbitrary paths in which no elementary move is realized by an special couple.

An important argument in the proof of the next Proposition is the fact that if an elementary move $\nu_{0}\to \nu_{1}$ is realized by the curves $\alpha$ and $\beta$, the boundary curves of a regular neighborhood $N$ of $\alpha \cup \beta$ belong to both $\nu_{0}$ and $\nu_{1}$. Then in order to measure the distance between $\pi_{Q}(\nu_{0})$ and $\pi_{Q}(\nu_{1})$ we have just to look at those subsurfaces $Y_{i}\subset \Sigma\setminus Q$ where $\alpha$ and $\beta$ have nonempty intersection once they are tight. Any single arc of either $\alpha$ or $\beta$ intersecting $Y_{i}$ implies the existence of at least another isotopic to it in $Y_{i}$ as the  result of intersecting curve in $\partial N$ with $Y_{i}$, thus an arc in $\nu_{0}\cap \nu_{1}$ makes no increasing distance in the factor $\mathcal{P}_{Y_{i}}$.

\begin{prop}\label{pro1} Let $\boldsymbol{\nu}=(\nu_{0},\cdots ,\nu_{r})$ be a path in the pants graph $\mathcal{P}$ such that no pair of curves $\alpha, \beta \subset \Sigma$ realizing an elementary move $\nu_{i}\to\nu_{i+1}$ intersects a complexity $1$ subsurface $Y\subset \Sigma \setminus Q$ forming a special couple $(\alpha \cup \beta)\cap Y$. Then there exist $\omega_{0} \in \pi_{Q}(\nu_{0})$ and $\omega_{r} \in \pi_{Q}(\nu_{r})$ such that 
\[
d_{Q}(\omega_{0},\omega_{r})\leq r
\]
\end{prop}
\begin{proof} All arcs considered in this proof inside a complexity $1$ surface $Y$ are assumed to be the result of intersecting with $Y$ curves  that are tight in it. 

For each two adjacent complexity $1$ subsurfaces $Y_{i},Y_{j}\subset \Sigma\setminus Q$ there is at most one single common boundary component, due to the homology condition in the definition of any such $Y_{k}$ given in \ref{mc}. For such a common boundary we shall consider a regular neighborhood of it, that is, an annulus $A_{ij}$ in such way that its boundary circles become boundaries of $Y_{i}$ and $Y_{j}$. All along this proof the curves we are dealing with, are  tight in all $Y_{i}, i=1,\dots,n$, eventually deforming them isotopically to make some intersection points between two of them lying outside any $Y_{i}$. 

Consider first those subpaths $\boldsymbol{\nu}_{\mathbf{1}},\dots,\boldsymbol{\nu}_{\mathbf{s}}\subset\boldsymbol{\nu}$ such that all elementary moves inside them are of the second kind and each pair of curves realizing such move have their intersection points at different complexity $1$ subsurfaces. 

We shall deal first with projections of this class of subpaths to control their distance after projecting to $\mathcal{P}_{Q}$. 

Let $\boldsymbol{\nu}_{j}\subset \boldsymbol{\nu}$ one of the above subpaths and assume first that $\mathrm{length}(\boldsymbol{\nu}_{j})\geq 3$. Set $\nu_{1},\dots ,\nu_{k}$ the set of vertices of $\boldsymbol{\nu}_{j}$. The elementary move $\nu_{1}\to \nu_{2}$ is realized by a pair of curves, say $\alpha$ and $\beta$ with their intersection points contained in $Y_{1}\cup Y_{2}\subset \Sigma\setminus Q$. Let $N\subset Y_{1}$ be a regular neighborhood of the component of $(\alpha \cup \beta)\cap Y_{1}$ that contains the intersection point. Call $d_{1},\dots ,d_{4}$ the arcs from $\partial N \cap Y_{1}$. We stress that these arcs are the result of intersecting four different curves in $\nu_{1}\cap \nu_{2}$ with $Y_{1}$.

In case of $Y_{1}$ being a $1$-holed torus, Corollary \ref{3atm} ensures that two of the arcs in $\partial N \subset Y_{1}$ occupy one isotopy class while the other two other isotopy class, being these two classes different from those of $\alpha \cap Y_{1}$ and $\beta \cap Y_{1}$, so an eventual curve $\gamma \in \nu_{i+2}\subset \boldsymbol{\nu}_{j}$ intersecting a curve in $\nu_{i}$ inside $Y_{1}$ could only do it with $\alpha$ again as an intersection with an arc $d_{k}\in \partial N$ would imply the intersection with its isotopic fellow $d_{s}$ both coming from different curves in $\nu_{i+1}$. Due to this fact, further analysis on intersection of arcs take us to deal with the case of $Y_{i}$ being a $4$-holed sphere.

We may select one of $d_{1},\dots ,d_{4}$, say $d_{1}$ to construct $\pi_{Y_{1}}(d_{1})$ as that curve in $Y_{1}$ that belongs to the projection of $\nu_{1},\nu_{2},\nu_{3}$, and eventually of $\nu_{4}$ and $\nu_{5}$, depending on the length of $\boldsymbol{\nu}_{j}$. By proceeding analogously on all that subsurfaces $Y_{l}\subset \Sigma\setminus Q$ containing point of intersection between curves that are involved in an elementary move inside $\boldsymbol{\nu}_{j}$, we can ensure 
\begin{equation}\label{e1}
d_{Q}(\pi_{Q}(\nu_{1}),\pi_{Q}(\nu_{i}))=0,\;\; i\leq 5
\end{equation}

In the particlar case that $\mathrm{length}(\boldsymbol{\nu}_{j})= 2$ we may select curves as before to have $d_{Q}(\pi_{Q}(\nu_{1}),\pi_{Q}(\nu_{3}))=0$, even in this particular case where $\nu_{3}\not\subset \boldsymbol{\nu}_{j}$.

\begin{figure}[h]
\begin{center}
\scalebox{0.22}{\includegraphics{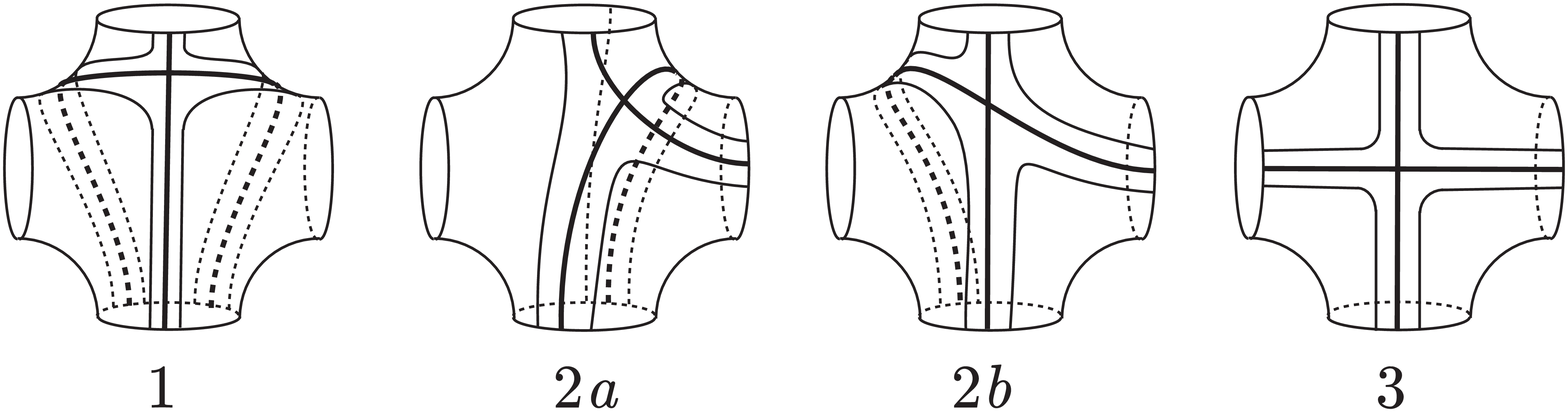}}\label{comf}
\end{center}
\caption{The figure illustrates the proof of proposition \ref{pro1}. The portions of $\alpha$ and $\beta$ contained in $Y$, in bold line in, intersect in one single point. Also the corresponding arcs in $\partial N$ are shown.}
\end{figure}

To follow the proof we observe now that there are four essentially different possibilities for two arcs contained in a $4$-holed sphere $Y$ (up to homeomorphism of $Y$) intersect in one point attending to the number of components of $\partial Y$ containing end points of the two arcs, and how many of them are waves or seams, these are drawn in figure $3$.

Next we are concerned with an eventual need of change of the selected arc $d_{1}\subset Y_{1}$ to project vertices $\nu_{i}$ in $\boldsymbol{\nu}_{j}, i\geq 5$. This situation occurs when $d_{1}$ is part of a curve involved the elementary move $\nu_{5}\to \nu_{6}$ with just one intersection point in $Y_{1}$. Call $f_{1}$ the arc that intersects $d_{1}$ in this elementary move and let $N\subset Y_{1}$ be a regular neighborhood of $d_{1}\cup f_{1}$. Then as shows figure $3$, there are at least two arcs $d'_{1},d'_{2}\subset \partial N$ such that $d_{Y_{1}}(\pi_{Y_{1}}(d_{1}),\pi_{Y_{1}}(d'_{i}))=1$. The least favorable case would be that in the same $\nu_{5}\to\nu_{6}$, there is an arc $f_{2}\subset Y_{2}$ for another $4$-holed sphere that intersects the selected arc $d_{2}\subset Y_{2}$ in one single point, that is $\pi_{Y_{2}}(d_{2})=\pi_{Y_{2}}(\nu_{5})$. We may assume that both $d_{1}$ and $d_{2}$ were used to construct $\pi_{Y_{1}}(\nu_{i})$ and $\pi_{Y_{2}}(\nu_{i})$ for $i=1,\dots ,5$. The existence of at least two such $d'_{1}$ and $d'_{2}$ leads us to ensure $d_{Q}(\pi_{Q}(\nu_{5}),\pi_{Q}(\nu_{7}))\leq 2$, also considering the case that $\nu_{7}\notin \boldsymbol{\nu}_{j}$. This sums to the bound in (\ref{e1}) to get $d_{Q}(\pi_{Q}(\nu_{1}),\pi_{Q}(\nu_{7}))\leq 2$, considering also the case of $\nu_{7}\not\in\boldsymbol{\nu}_{j}$. 

Continuing this process we get for a subpath $\boldsymbol{\nu}_{j}$ having length $l$
\begin{equation}
d_{Q}(\pi_{Q}(\nu_{1}),\pi_{Q}(\nu_{l+2}))\leq l-3,
\end{equation}
where $\nu_{l+2}\not\in \boldsymbol{\nu}_{j}$.

Now let $\nu_{0}\notin\boldsymbol{\nu}_{j}$ be a vertex that is adjacent to $\nu_{1}\in\boldsymbol{\nu}_{j}$. Call $\lambda,\mu\subset \Sigma$ the curves realizing $\nu_{0}\to\nu_{1}$. This elementary move may be either of the first or of the second kind, in any case their intersection points do not lie at different complexity $1$ subsurfaces on $\Sigma \setminus Q$. 

In order to attach the projections of all the paths $\boldsymbol{\nu}_{j}$ along those ones of vertices that do not belong to them, we deal firstly with the case that $\nu_{0}\to\nu_{1}$. Let $d_{0}\subset Y_{1}$ be the arc or curve considered to make $\pi_{Y_{1}}(\nu_{0})$. Let $d_{1}\subset Y_{1}$ be the arc whose projection is $\pi_{Y_{1}}(\nu_{1})$, as explained above. In this case $Y_{1}$ contains one point of $\alpha \cap \beta$, the possible need to change  $d_{0}$ could be due to two reasons: one is because this arc intersects $d_{1}$ and the other because it does not coincide with it. Notice $d_{0}$ can not intersect either $\alpha$ or $\beta$ unless it is contained in $N$ and for this reason it can not intersect $d_{1}$ as both $d_{0}$ and $d_{1}$ are assumed to be tight in $Y_{1}$ and $N\subset Y_{1}$ is homeomorphic to a disc. The remaining possibility is $d_{0} \subset Y_{1}\setminus N$. In case that $Y_{1}$ is a $1$-holed torus $d_{0}$ is an arc and by Lemma \ref{3atm} we have $d_{Y_{1}}(\pi_{Y_{1}}(d_{0}),\pi_{Y_{1}}(d_{i}))\leq 1$.  
In case that $Y_{1}$ is a $4$-holed sphere we may assume that $\pi_{Y_{1}}(d_{0})=\pi_{Y_{1}}(s)$ for a seam $s\subset Y_{1}\setminus N$, then by Lemma \ref{int} we have $d_{Y_{1}}(\pi_{Y_{1}}(d_{0}),\pi_{Y_{1}}(d_{i}))\leq 2$, and so we may ensure that $d_{Q}(\pi_{Q}(\nu_{0}),\pi_{Q}(\nu_{1}))\leq 2$.
 
Finally let $\nu_{r}\to \nu_{r+1}$ be any edge that does not lie in $ \boldsymbol{\nu}_{k}, k=1,\dots ,s$ and does not coincide with any of those used in the construction above.  Let $\xi, \zeta$ be the curves that realize this move. These curves intersect in a single subsurface $Y_{0}\subset \Sigma \setminus Q$. Let $d_{r}\subset Y_{0}$ be the arc or curve selected to make $\pi_{Y_{0}}(\nu_{r})$, then if $\zeta$ does not intersect $d_{r}$ we still consider this projection as contained in $\pi_{Y_{0}}(\nu_{r+1})$, otherwise consider the regular neighborhood $N$ of $d_{r}\cap\zeta\cap Y_{0}$ and by Lemmas \ref{prt} and \ref{ml} we have $d_{Y_{0}}(\pi_{Y_{0}}(d),\pi_{Y_{0}}(d_{r}))\leq 1$ for some arc $d\in \partial N$ then we have $d_{Q}(\pi_{Q}(\nu_{r}),\pi_{Q}(\nu_{r+1}))\leq 1$.

Assembling the bounds for the subpaths $\boldsymbol{\nu}_{j}$ together with this last bound for single edges not inside them gives the result.
\end{proof}

\begin{lem}[Orthogonality] \label{l2}Let $Q$ be a $n$-handle multicurve and let $\nu_{0}\in \mathcal{P}_{Q}$ and $\nu_{1}\notin \mathcal{P}_{Q}$ be two adjacent vertices in $\mathcal{P}$. Then $\pi_{Q}(\nu_{1})=\nu_{0}$.
\end{lem}
\begin{proof} Assume the elementary move $\nu_{0}\to \nu_{1}$ is realized by the curves $\alpha \in \nu_{0}$ and $\beta \in\nu_{1}$.

It may happen that $\beta$ does not intersect any complexity $1$ subsurface $Y_{i} \in \Sigma \setminus Q$, but in this case clearly we have that $\pi_{Q}(\nu_{1})=\nu_{0}$ as we do not have to consider projections of curves over the $Y_{i}$. The other possibility is that $\beta$ intersects some $Y_{i}$ but it can not be contained in any one of them for in that case $\nu_{1}\in \mathcal{P}_{Q}$.

That said, computing $\pi_{Q}(\nu_{1})$ relies on computing $\pi_{Y_{i}}(\beta \cap Y_{i})$ for each $Y_{i}\subset \Sigma \setminus Q$ that $\beta$ intersects. The number of such subsurfaces is at most $2$ in case of $\alpha$ separates two complexity $1$ subsurfaces in $\Sigma \setminus Q$.

To finish the proof we assume that $\beta$ intersects $Y_{1} \subset \Sigma \setminus Q$. As $\kappa (Y_{1})=1$, there is a curve $\gamma \subset Y_{1}$ that belongs to $\nu_{0}\cap \nu_{1}$ and on the other hand $\beta \cap Y_{1}$ consists of one single arc, say $\beta_{1}$ as necessarily the curve $\alpha$ is one boundary component of $Y_{1}$, and $\pi_{Y_{i}}(\beta_{1})=\gamma$. Making the same reasoning for another eventual $Y_{2}$ proves the assertion.
\end{proof}

\mathstrut

\begin{lem}\label{LF} Let $\boldsymbol{\nu}=(\nu_{0},\dots ,\nu_{r})$ be a geodesic path in the pants graph $\mathcal{P}$. Let $\alpha$ and $\beta$ be a pair of curves realizing the elementary move $\nu_{k} \to \nu_{k+1}$ such that for some subsurface $Y\subset \Sigma \setminus Q$ their intersections with $Y$ form a special couple. Then there is at most another curve $\gamma \in  \nu_{k-1}\cup \nu_{k+2} $ such that $\gamma \cap Y$  forms a different special couple from the former, with just one of  $\alpha \cap Y$ or $\beta \cap Y$. 
\end{lem}

\begin{proof} We set $s$ be the seam resulting from the intersection $\alpha \cap Y$ an we notice that there is no restriction on assuming that $\alpha \in \nu_{k}$ as the path may be reversed. 

The curve $\beta \in \nu_{k+1}$ is contained in the $4$-holed sphere $Y$ and it induces a pants decomposition on $Y$, then $Y\setminus \beta$ is a couple of $3$-holed spheres, say $P_{1}$ and $P_{2}$. Now the seam $s$ intersects $\beta$ in two points and then it has its end points at, say $P_{1}$. 
 
Any curve $\gamma \in \nu_{k-1}$ such that $\gamma \cap Y$ forms a special couple with $\alpha \cap Y$ has to be contained in $Y$ and $|\gamma \cap s|=2$. We may consider $\gamma=\pi_{Y}(u)$, for a seam $u$ and by Lemma \ref{int} there are two possibilities for $u$, namely those giving $\pi_{Y}(u)=\beta$. The choice for $u$ is unique and so is its projection $\gamma$ that it would be equal $\beta$, thus in contradiction with $(\nu_{0},\dots ,\nu_{r})$ being a geodesic path.

On the other side, for a curve $\gamma \in \nu_{k+2}$ form a special couple with $\beta$ let $s'\subset \gamma \cap Y$ the corresponding seam, $s'$ have its end points either at $\partial P_{1}$ to avoid intersections with $\partial N(\alpha \cup \beta)\cap Y$. Then the possibilities for $s'$ are all powers of half-twists along $\beta$, of $\alpha \cap Y$. These are the possibilities for $\gamma \cap Y$ form a special couple with $\beta$, and as before there is only one $\delta \in \nu_{k+3}$ that would form a special couple with $\gamma$ in $Y$, namely the same $\beta$ and in that case we get again the same contradiction.
\end{proof}

\mathstrut

\begin{lem}\label{sc} Let $\boldsymbol{\nu}=(\nu_{0},\dots ,\nu_{r})$ be a geodesic path in the pants graph $\mathcal{P}$. Let $\nu_{k}\to \nu_{k+1}$ be an elementary move such that the curves $\alpha$ and $\beta$ realizing this elementary move are special for a complexity $1$ subsurface $Y\subset \Sigma$. Assume that $\beta\subset Y$, then $\beta\in \pi_{Y}(\alpha)$.   
\end{lem}
\begin{proof} Notice first that the homology condition for the boundary curves of any complexity $1$ subsurface $Y\subset \Sigma\setminus Q$ implies that these curves are separating in $\Sigma$. Due to this fact, given a special couple $s\cup \beta \subset \Sigma$ where $s\subset \alpha \cap Y$ is a seam, there is another seam $s'\subset \alpha \cap Y$ that has its end points at the same components as $s$ and avoiding any intersection with $\beta$. It results that $\pi_{Y}(s')=\beta$ and $\beta\in \pi_{Y}(\alpha)\cap \pi_{Y}(\beta)$.  
\end{proof}

\mathstrut

\begin{thm}\label{TP} Let $\boldsymbol{\nu}=(\nu_{0},\dots, \nu_{r})$ be a geodesic path in the pants graph $\mathcal{P}$, then there exist $\omega_{0}\in\pi_{Q}(\nu_{0})$ and $\omega_{r}\in\pi_{Q}(\nu_{r})$ such that 
\[
d_{Q}(\omega_{0},\omega_{r})\leq r.
\]
\end{thm}
\begin{proof} Call $\boldsymbol{\nu}^{*}\subset \boldsymbol{\nu}$ a subpath such that for any edge $\nu_{k}\to\nu_{k+1}$ contained on it, the curves $\alpha$ and $\beta$ that realize this elementary move are special for some subsurface $Y\subset \Sigma \setminus Q$. As in the proof of Proposition \ref{pro1} we are concerned with those $Y\subset \Sigma \setminus Q$ in which there are intersection points of curves realizing the elementary moves. In the case of $\boldsymbol{\nu}^{*}$ may be more than one such subsurface where a special pair takes place and for each $Y$ the length of the subpath $\boldsymbol{\nu}^{*}_{Y}\subset \boldsymbol{\nu}^{*}$ is at most $2$ by Lemma \ref{LF}. Consider one such $Y\subset \Sigma \setminus Q$ and let $\alpha, \beta$ and eventually $\gamma$ be the curves involved in such subpath. We may assume that $\alpha \cap Y$ is a seam an so $\beta \subset Y$, then let $N$ be a regular neighborhood of $(\alpha \cup \beta)\cap Y$, that in case of length $2$ it also contains $\gamma$. Exactly two components of $\partial N$, $d_{1}$ and $d_{2}$ are isotopic to other two components $c_{1},c_{2}\subset\partial Y$ and the other two $d_{3}$ and $d_{4}$ are seams isotopic between them and joining the other two components $c_{3},c_{4}\subset \partial Y$. As in Lemma \ref{sc} we may select one, $d_{3}$, and consider $\pi_{Y}(d_{3})$ as the corresponding curve for the projection of $\boldsymbol{\nu}^{*}_{Y}$.

When $\boldsymbol{\nu}^{*}_{Y}$ has length $1$ we may assume $\alpha \cap Y$ being a seam. Consider all possible seams from any of $c_{1}, c_{2}$ to one of $c_{3}, c_{4}$ such that intersect $\alpha \cap Y$ at most once. Now for one such seam $s$ intersecting $\alpha \cap Y$ consider the wave $w$ that has $s$ as an associated seam and avoids any intersection with $c_{1}$ or $c_{2}$. The set formed by the waves thus obtained together with those seams that do not intersect $\alpha$, $\alpha \cap Y$ and $d_{3}$ are all possibilities for arcs to be projected as $\pi_{Y}(\nu_{k-1})\subset \pi_{Q}(\nu_{k-1})$. Apart from $d_{Y}(\pi_{Y}(\alpha \cap Y),\pi_{Y}(d_{3}))=2$, for all other pair of arcs we have that this distance is $1$.

If we impose the arc $d_{3}$ as to be projected from the vertices of $\boldsymbol{\nu}^{*}_{Y}$ we may find a discrepancy between the one coming from the previous vertices as constructed in Proposition \ref{pro1}. We now measure this discrepancy in terms of distance between the projections. Notice that together with $d_{3}$ and $d_{4}$ there is another isotopic arc to them, say $d_{5}\subset \alpha \subset \nu_{k}$. 

As all the arcs $d_{3},d_{4}$ and $d_{5}$ correspond to different curves, we may assume that al least one of them, say $d_{3}$ still remains at the position $\nu_{k-2}$ but perhaps not in $\nu_{k-3}$ in this least favorable case, as a result of the construction in Proposition \ref{pro1}, either another arc or curve in the position of $\nu_{k-3}$ intersects $d_{3}$ or it is in the boundary of a region $R$ as that in figure $3$ or it lies outside $R$. In the first case the arc $d$ provided by Lemma \ref{ml} in the referred construction verifies $d(\pi_{Y}(d),\pi_{Y}(d_{3}))\leq1$. In case that $d_{3}$ lies outside $R$ is checked by inspection that we have the same bound between the projections of $d_{3}$ and any arc of $\partial R$. Now in case that $d_{3}$ is one of the arcs in $\partial R$ we may take it as the selected arc in that construction. The same reasoning holds for $\boldsymbol{\nu}^{*}_{Y}$ having length $2$.

Projecting $d$ if needed, from the position $\nu_{k-3}$, and $d_{3}$ from the rest, we obtain $d(\pi_{Y}(\nu_{k-3}), \pi_{Y}(\nu_{k+1}))\leq 1$ which is to say that we get that the subpath $\boldsymbol{\nu}^{*}_{Y}$ only increases distance by $1$. As this attaching is done by considerations on the subsurface $Y$, for any other $\boldsymbol{\nu}_{Y'}^{*}\subset \boldsymbol{\nu}^{*}$ that overlaps the attaching construction just made from $\nu_{k-3}$ to $\nu_{k+1}$ or $\nu_{k+2}$ allows us to keep the distance bound for the whole projections, thus $d(\pi_{Q}(\nu_{k-3}),\pi_{Q}(\nu_{k+1}))\leq 4$ or $d(\pi_{Q}(\nu_{k-3}),\pi_{Q}(\nu_{k+2}))\leq 5$ depending on the length of $\boldsymbol{\nu}^{*}_{Y}$. 

Finally, we apply the result of \ref{pro1} to those subpaths of $\boldsymbol{\nu}$ in the hypothesis of that Lemma, eventually with the correction to attach any $\boldsymbol{\nu}{*}$ to get the result. 
\end{proof}

\mathstrut

\section{Proof of theorem \ref{T1}}

The proof of Theorem \ref{T1} as in the case of $2$-handle multicurves \cite{APK2} relies on the orthogonality property of lemma \ref{l2}
\begin{proof}[Proof of theorem \ref{T1}] Let $\boldsymbol{\nu}=(\nu_{0},\dots, \nu_{r})\subset \mathcal{P}$ be a geodesic path with their end vertices at $\mathcal{P}_{Q}$. Assume that $\boldsymbol{\nu '}=(\nu_{1},\dots,\nu_{r-1})\not\subset\mathcal{P}_{Q}$. The path $\boldsymbol{\nu '}$ is still geodesic and by Theorem  \ref{TP} there are $\omega_{1} \in \pi_{Q}(\nu_{1})$ and $\omega_{r-1} \in \pi_{Q}(\nu_{r-1})$ with $d_{Q}(\pi_{Q}(\nu_{1}),\pi_{Q}(\nu_{r-1}))\leq r-2$. Besides, by orthogonality $\pi_{Q}(\nu_{1})=\nu_{0}$ and $\pi_{Q}(\nu_{r-1})=\nu_{r}$. Thus we would have $d_{Q}(\nu_{0},\nu_{r})\leq r-2$ contradicting the fact of $\boldsymbol{\nu}$ being geodesic.
\end{proof}

\mathstrut

\section{Large Flats in $\mathcal{P}$}\label{LF}

Disjointness of the complexity $1$ subsurfaces that are defined by an $n$-handle multicurve $Q$ allows us to consider a cartesian structure on the subgraph $\mathcal{P}_{Q}$. The proof of the following result is an extension to the case of $n$ such subsurfaces of Lemma 7 in \cite{APK2}.

\begin{lem} Let $Q$ be an $n-$handle multicurve. The graph $\mathcal{P}_{Q}$ is isomorphic to the product $\mathcal{P}_{Y_{1}}\ast \cdots \ast \mathcal{P}_{Y_{n}}$ of $n$ Farey graphs.\qed
\end{lem} 

If $Q$ is any $n-$handle multicurve then as a result of theorem \ref{T1} we have the following

\begin{cor} Let $n$ an integer for which $\Sigma$ admits an $n-$handle multicurve. Then $\mathcal{P}$ admits an isometric embedding of $\mathbb{Z}^{n}$.
\end{cor}
\begin{proof} Let $Q$ be a $n-$handle multicurve in $\Sigma$ and let 
\[
Y_{1},\dots , Y_{n}
\]
be the complexity $1$ surfaces in $\Sigma \setminus Q$. Choose a totally geodesic, bi-infinite, geodesic $g_{i}$ contained in each Farey graph $\mathcal{P}_{Y_{i}}$. 

The product graph $g_{1}\ast \cdots \ast g_{n}$ embeds in $\mathcal{P}_{Q}$ as a totally geodesic subgraph, as $\mathcal{P}_{Q}$ is totally geodesic in $\mathcal{P}$, it results that the former is totally geodesic in $\mathcal{P}$.
\end{proof}

\mathstrut

The surface $\Sigma$ may be decomposed in subsurfaces
\[
Y_{1}, \dots , Y_{r(\Sigma)},T
\] 
with $r(\Sigma)=\lfloor \frac{\kappa(\Sigma)+1}{2} \rfloor$ and $Y_{i}$ being of complexity $1$ and either $T$ the empty set or a pair of pants \cite{Br2, Br-F}. 

The curves defining this decomposition are in case that $T$ is vacuous a $n-$handle multicurve if we abide by the definition. Nevertheless, we may consider the case in which $T$ exists by taking into account that the projection of any curve over $T$ is a curve parallel to the boundary and so considering the set $\pi_{T}(a)$ the empty set for any curve $a\subset \Sigma$.

In this case $\mathcal{P}_{Q}$ is a product of $n=r(\Sigma)$ Farey graphs. Let $L$ be the Weil-Petersson length of an edge in one of these factor graphs. We know that in case that the factor graph corresponds to a $4$-holed sphere, the value of $L$ is twice that one of in case of a $1$-holed torus \cite{Br-M}. Once we have the embedding $\mathbb{Z}^{n}$ we may construct by rescaling the euclidean metric in the appropriate directions, an isometric embedding of $\mathbb{R}^{n}$ in $\mathcal{P}_{Q}$ and then in $\mathcal{P}$, where $n$ is the maximum possible \cite{W1,Y}. 

\section*{Acknowledgments}

I thank Hugo Parlier interesting comments and essential corrections on previous versions of this manuscript.

\bibliographystyle{amsplain}

\begin{thebibliography}{99}




\bibitem{APK1} J. Aramayona, H. Parlier and K. J. Shackleton,
\emph{Totally geodesic subgraphs of the pants complex} Math. Res. Lett. \textbf{15}, no. 2-3 (2008), 309-320.

\bibitem{APK2} J. Aramayona, H. Parlier and K. J. Shackleton,
\emph{Constructing convex planes in the pants complex} Proc. A. M. S. \textbf{137} (10) 3523-3531.


\bibitem{ALPK} J. Aramayona, C. Lecuire, H. Parlier and K. J. Shackleton,
\emph{Convexity of strata in diagonal pants graphs of surfaces} Preprint.
arXiv:1111.1187 [math. GT].


\bibitem{B-M} J. A. Behrstock and Y. N. Minsky, \emph{Dimension and rank for mapping class groups}. Ann. Math. \textbf{167}, no. 3 (2008), 1055-1077.


\bibitem{Br1} J. F. Brock, \emph{The Weil-Petersson metric and volumes of 3-dimensional convex cores}. J. Amer. Math. Soc. \textbf{16}, no. 3 (2003), 495-535.


\bibitem{Br2} J. F. Brock, \emph{Pants decomposition and the Weil-Petersson metric}. Contemp. Math. \textbf{311}, (2002), 27-40.


\bibitem{Br-F} J. F. Brock and B. Farb, \emph{Curvature and rank of Teichm\"uller space}. Amer. J. Math. \textbf{128}, (2006), 1-22.

\bibitem{Br-M} J. F. Brock and D. Margalit, \emph{Weil-Petersson isometries via the pants complex}. Proc. Amer. Math. Soc. \textbf{135}. no. 3 (2007), 795-803.

\bibitem{C-P} W. Cavendish and H. Parlier, \emph{Growth of the Weil-Petersson diameter of Moduli space}. Duke Math. J. \textbf{161} (2012), no. 1, 139-171. 


\bibitem{H} A. Hatcher, P. Lochak and L. Schneps. \emph{On the Teichm\"uller tower of mapping class groups} J. reine angew. Math. \textbf{521} (2000), 1-24.


\bibitem{HT} A. Hatcher and W. Thurston, \emph{A presentation for the mapping class group of a closed orientable surface}. Topology \textbf{19} (1980), 221-237.

\bibitem{W1} S. Wolpert
\emph{On the geometry of the Weil-Petersson completion of Teichm\"uller spaces.} Preprint


\bibitem{Y} S. Yamada
\emph{Geometry of the Weil-Petersson completion of Teichm\"uller space.} Math. Res. Lett. \textbf{11}, (2004), 327-344.


\end{thebibliography}

\end{document}